\documentclass[12pt,letterpaper]{amsart}
\usepackage{geometry}
\usepackage{graphicx}
\usepackage{amsfonts}

\usepackage{mathrsfs}
\usepackage{amssymb,latexsym}
\usepackage[dvips]{color}

\newtheorem{theorem}{Theorem}

\newtheorem{proposition}{Proposition}

\theoremstyle{remark}

\newcommand{\abs}[1]{\left\vert#1\right\vert}
\newcommand{\norm}[1]{\Vert#1\Vert}

\newcommand{\C}{\mathbf{C}}

\providecommand{\bysame}{\leavevmode\hbox to3em{\hrulefill}\thinspace}
\providecommand{\MR}{\relax\ifhmode\unskip\space\fi MR }

\providecommand{\href}[2]{#2}

\begin{document}

\title[Compactness in the $\overline{\partial}$-Neumann problem]{Observations regarding compactness in the $\overline{\partial}$-Neumann
problem}
\author{Mehmet \c{C}el\.{i}k \& Emil J. Straube}

\address{Department of Mathematics\\ Texas A\&M University\\
               College Station, TX 77843}

\email{celik@math.tamu.edu, straube@math.tamu.edu}

\thanks{Work supported in part by NSF grants DMS-0500842 and DMS-0758534}


\maketitle

\section{Introduction}
Compactness of the $\overline{\partial}$-Neumann operator plays an
important role in several contexts. The condition was initially
introduced by Kohn and Nirenberg \cite{MR0181815} as a sufficient
condition for global regularity. Work of Catlin \cite{MR740870} and
Sibony \cite{MR894582} showed that this does indeed provide a viable
route to global regularity; that is, the compactness condition can
be verified on large classes of domains. We refer the reader to
\cite{MR1748601, MR1800297, MR1912737, MR2275654, Vienna} for
background on the $\mathcal{L}^{2}$-Sobolev theory of the
$\overline{\partial}$-Neumann problem in general and on compactness
in particular.

Subellipticity of the $\overline{\partial}$-Neumann operator is
independent of the metric (a fact usually attributed to Sweeney
(\cite{MR0298708}), although that reference deals with the coercive
case), while continuity in Sobolev spaces is not. The latter is a consequence
of Kohn's results concerning estimates for the
$\overline{\partial}$-Neumann problem with weights
(\cite{MR0344703}) and Barrett's results on failure of Sobolev estimates on
the worm domains (\cite{MR1149863}). Compactness is intermediate between
subellipticity and continuity in Sobolev spaces. Consequently, it is of
interest to know that compactness is also independent of the metric, for metrics subject
only to the condition that they be smooth on the closure of the
domain (so that the induced norms on the $L^{2}$-spaces of forms are
equivalent). In particular, the metrics at higher form levels are
not required to be induced by the metric on $(0,1)$-forms. We also
obtain a new proof of the independence of subellipticity of the
$\overline{\partial}$-Neumann operator from the metric (for the same
class of metrics).

We define the notion of a compactness multiplier in obvious analogy
to that of a subelliptic multiplier (\cite{MR512213, MR1748604}). It
is easily seen that the continuous multipliers form an ideal of the
form $\{f \in C(\overline{\Omega})\;|\;f(z)=0, z \in A\}$, where the
common zero set $A$ is a compact subset of the boundary. This common
zero set may thus be viewed as the obstruction to compactness: the
$\overline{\partial}$-Neumann operator is compact if and only if $A$
is the empty set. Of course, this purely abstractly defined
obstruction is of use only to the extent that it can actually be
determined if the domain is given. We do this for two classes of
domains where compactness is understood, i.e. convex domains in
$\mathbb{C}^{n}$ and complete Hartogs domains in $\mathbb{C}^{2}$.
For a convex domain, and for $(0,1)$-forms, this set is the closure
of the union of all the analytic discs in the boundary. For a
Hartogs domain in $\mathbb{C}^{2}$, on the other hand, the closure 
of the union of all the analytic discs in the boundary
can be strictly contained in $A$. This is a reflection of the fact 
that even in
$\mathbb{C}^{2}$, there can be obstructions to compactness more
subtle than analytic discs in the boundary (\cite{Matheos}, \cite{
MR1912737}, Theorem 4.2). What matters are fine interior points of the
projection (on the base) of the weakly pseudoconvex points, not just
Euclidean interior points of this projection (which correspond to 
analytic discs in the
boundary, cf. \cite{Vienna}, Lemma 3.18 ). Accordingly, $A$ equals
the Euclidean closure of the inverse image of these fine interior
points (under a mild technical condition).

Much of this material comes from the first author's Ph.D.
dissertation (\cite{celik}) written at Texas A\&M University under
the supervision of the second author.

\section{Variation of the Metric}
A theorem by W. J. Sweeney in \cite{MR0298708} shows that coercive
estimates are independent of the metric on the tangent bundle, and
the fact that the same is true for subelliptic
estimates is also usually attributed to him. In view of Kohn's results concerning Sobolev estimates
for the $\overline{\partial}$-Neumann operator associated to
suitably weighted metrics (\cite{MR0344703}), and Barrett's results
on failure of Sobolev estimates on the worm domains
(\cite{MR1149863}), this invariance does not hold for Sobolev
estimates. As compactness is a regularity property that lies between
subellipticity and continuity in Sobolev spaces, it is natural to
ask how it behaves when the metric is changed, and we note that it is independent of the metric. We then
give a new proof that subellipticity of the $\overline{\partial}$-Neumann operator is independent of the metric.

Let $\Omega$ be a bounded pseudoconvex domain in $\mathbb{C}^{n}$. For $q=1, \cdots, n$, denote by
$G^{q}(z)=G^{q}_{I,J}(z)$ a smooth function on $\overline{\Omega}$ with values in the strictly positive definite
Hermitian $\left (n!/(n-q)!q! \right )$ by $\left (n!/(n-q)!q! \right )$  matrices; $I$ and $J$ are strictly increasing $q$-tuples. We denote by
$L^{2}_{(0,q)}(\Omega, G^{q})$ the square-integrable $(0,q)$-forms, but with the standard inner product
replaced by the one given pointwise by $G^{q}_{I,J}(z)$. Let
$u=\sum_{\abs{J}=q}'u_{J}d\overline{z}_{J}$, $v=\sum_{\abs{K}=q}'v_{K}d\overline{z}_{K}$. Here, the prime
indicates as usual summation over increasing $q$-tuples, and
$d\overline{z}_{J} = d\overline{z}_{j_{1}}\wedge \cdots \wedge d\overline{z}_{j_{q}}$. Then
\begin{equation}\label{eq1}
(u,v)_{G^{q}}:=\int_{\Omega}\langle u,v\rangle_{G^{q}}
=\sum_{\abs{J}=\abs{K}=q}' \int_{\Omega} G^{q}_{JK}(z)u_{J}(z)\overline{v_{K}(z)}dV \; ,
\end{equation}
and
\begin{equation}\label{eq2}
\norm{u}_{G^{q}}^{2}:=\int_{\Omega}\langle u,u\rangle_{G^{q}} =\sum_{\abs{J}=\abs{K}=q}'\int_{\Omega} G^{q}_{JK}u_{J}(z)\overline{u_{K}(z)}dV.
\end{equation}
Note that since $G^{q}_{I,J}$ is smooth on the closure of $\Omega$ and takes only positive definite  matrices as values,
the \emph{set} of square-integrable forms does not change, but the inner product does (to an equivalent one).
With this inner product, $L^{2}_{(0,q)}(\Omega,G^{q})$ is a Hilbert space. It will be convenient to refer to the inner product in \eqref{eq1} and the norm in \eqref{eq2} as `weighted'.

\emph{Remark 1}: The metrics $G^{q}$ for different $q$'s are not assumed related. In
particular, it is not assumed that $G^{q}$ is induced by $G^{1}$.

We briefly recall the setup for the $\overline{\partial}$-complex and the $\overline{\partial}$-Neumann problem in the case of a general metric. First,
\begin{eqnarray}\label{eq3}
\overline{\partial}u &=&\overline{\partial}\left( \sum_{\abs{J}=q}' u_{J}d\overline{z}_{J} \right)=\sum_{\abs{J}=q}'\overline{\partial} u_{J}\wedge d\overline{z}_{J}\\
&=&\sum_{j=1}^{n}\sum_{\abs{J}=q}' \frac{\partial u_{J}}{\partial\overline{z}_{j}}d\overline{z}_{j}\wedge d\overline{z}_{J} \; ,
\end{eqnarray}
where the domain consists of those forms where the right hand side, when computed as distributions, is in
$L^{2}_{(0,q+1)}(\Omega,G^{q+1})$. Thus $\overline{\partial}$ is closed and densely defined, at each level $q$, and so has a Hilbert space adjoint $\overline{\partial}^{*}_{G}$. For example, when $q=0$, the usual computation for this adjoint gives
\begin{eqnarray}\label{d-bar-star}
\left(\overline{\partial}\right)_{G}^{\star}u=-\sum_{j,k=1}^{n}\ (G^{0})^{-1}\frac{\partial \left(G^{1}_{jk}u_{j}(z)\right)}{\partial z_{k}}, \ u\in\mbox{Dom}\left(\left(\overline{\partial}\right)_{G}^{*}\right) \; ;
\end{eqnarray}
the boundary condition is
\begin{equation}\label{kkkk}
\sum_{j,k=1}^{n}\ G^{1}_{jk}\ u_{j}\frac{\partial \rho}{\partial z_{k}}  =\  0\ \ \mbox{for}\ \ z\in b\Omega \; .
\end{equation}
Note that if $G^{1}$ is the Euclidean metric, so that $G^{1}_{jk}=\delta_{jk}$ (where $\delta_{jk}$ is the Kronecker-$\delta$), we obtain the familiar boundary condition $\sum_{k=1}^{n} u_{k}\frac{\partial \rho}{\partial z_k}=0$ on $b\Omega$.

For $u,v\in
\mbox{Dom}(\overline{\partial}_{q})\cap\mbox{Dom}\left(\left(\overline{\partial}_{q-1}\right)_{G}^{\star}\right)$,
the Dirichlet form $Q_{G,q}(u,v)$ is defined as
\begin{eqnarray}
Q_{G,q}(u,v):=\left(\overline{\partial}u,\overline{\partial}v\right)_{G^{q+1}}+
\left(\left(\overline{\partial}_{q-1}\right)_{G}^{\star}u,\left(\overline{\partial}_{q-1}\right)_{G}^{\star}v\right)_{G^{q-1}}.
\end{eqnarray}
$\mbox{Dom}(\overline{\partial}_{q})\cap\mbox{Dom}\left(\left(\overline{\partial}_{q-1}\right)_{G_{q}}^{\star}\right)$
is complete with respect to $|||\ u\
|||_{G}^{2}:=Q_{G}(u,u)+\norm{u}_{G}^{2}.$ Thus there is a unique
non-negative selfadjoint operator $\Box_{q}^{G}$ associated to
$Q_{G,q}$ via
\begin{eqnarray}\label{ww}
Q_{G,q}(u,v)=(\Box_{q}^{G}u,v)_{G^{q}},\ \ u\in \mbox{Dom}(\Box_{q}^{G}),
\end{eqnarray}
where $\mbox{Dom}(\Box_{q}^{G})$ consists of those $u$ in
$\mbox{Dom}(\overline{\partial}_{q})\cap\mbox{Dom}\left(\left(\overline{\partial}_{q-1}\right)_{G}^{\star}\right)$
with $\overline{\partial}u \in
\mbox{Dom}\left(\left(\overline{\partial}_{q}\right)_{G}^{\star}\right)$
and $\left (\overline{\partial}_{q-1}\right )_{G}^{\star}u \in
\mbox{Dom}(\overline{\partial}_{q-1})$. (See \cite{MR751959},
Theorem VIII.15, or \cite{Davies} Theorem 4.4.2: if $Q$ is a closed
symmetric quadratic form then $Q$ is the quadratic form of a unique
self adjoint operator as in \eqref{ww}).

Because $G^{q+1}$ induces a norm on
$L^{2}_{(0,q+1)}(\Omega,G^{q+1})$ that is equivalent to the
Euclidean one, the domain and range of $\overline{\partial}$ are
unchanged from the Euclidean setting. In particular, the range is
still equal to the kernel of $\overline{\partial}$ acting on
$(q+1)$-forms. The range of
$\left(\overline{\partial}_{q}\right)_{G}^{\star}$ is then also
closed (because the range of $\overline{\partial}_{q}$ is, see for
example, Lemma $4.1.1$ in \cite{MR1800297}).  However, the range of
$\left(\overline{\partial}_{q}\right)_{G}^{\star}$ is also dense in
ker$(\overline{\partial}_{q})^{\bot_{G^{q}}}$, and so
$\mbox{Im}\left(\left(\overline{\partial}_{q}\right)_{G}^{\star}\right)=
\mbox{ker}(\overline{\partial}_{q})^{\bot_{G^{q}}}$. It follows that
\begin{eqnarray}\label{ssss}
L_{(0,q)}^{2}(\Omega,G^{q})=\underbrace{\mbox{ker}(\overline{\partial}_{q})}_{\mbox{Im}(\overline{\partial}_{q-1})}\oplus
\underbrace{\mbox{Im}\left(\left(\overline{\partial}_{q}\right)_{G}^{\star}\right)}_{\mbox{ker}\left(\left(\overline{\partial}_{q-1}\right)_{G}^{\star}\right)}
\end{eqnarray}
(since also $\mbox{Im}(\overline{\partial}_{q-1})^{\bot_{G^{q}}} = \mbox{ker}((\overline{\partial}_{q-1})^{*}_{G})$ ).
Therefore,  $\ker(\Box_{q}^{G})=
\ker(\overline{\partial}_{q})\cap\ker\left(\left(\overline{\partial}_{q-1}\right)_{G}^{\star}\right)=
\left\{0\right\}$ (the first equality is from \eqref{ww}). Then, Theorem $1.1.2$ in
\cite{MR0179443} implies
\begin{eqnarray}\label{seven}
\norm{u}_{G^{q}}^{2}\lesssim\norm{\overline{\partial}_{q}u}_{G^{q+1}}^{2}+
\norm{\left(\overline{\partial}_{q-1}\right)_{G}^{\star}u}_{G^{q-1}}^{2}\;
\end{eqnarray}
for
$u\in\mbox{Dom}(\overline{\partial})\cap\mbox{Dom}\left(\left(\overline{\partial}_{q-1}\right)_{G}^{\star}\right)$.
\eqref{seven} is the crucial estimate for the $L^{2}$-theory, and
general Hilbert space arguments now give as usual that $\Box_{q}^{G}$ has a
bounded inverse. In fact, if
$u,v\in\mbox{Dom}(\overline{\partial}_{q})\cap\mbox{Dom}\left(\left(\overline{\partial}_{q-1}\right)_{G}^{\star}\right)$,
then
\begin{equation}\label{7a}
\abs{(u,v)_{G^{q}}} \leq \norm{u}_{G^{q}}\norm{v}_{G^{q}} \lesssim
\norm{u}_{G^{q}}\left(\norm{\overline{\partial}_{q}v}_{G^{q+1}}^{2}+
\norm{\left(\overline{\partial}_{q-1}\right)_{G}^{\star}v}_{G^{q-1}}^{2}\right)^{\frac{1}{2}}.
\end{equation}
That is, the functional on the left hand side of \eqref{7a} is a
continuous (conjugate linear) functional in $v$ (for $u$ fixed) in the norm induced by
$Q_{G^{q}}$ on
$\mbox{Dom}(\overline{\partial}_{q})\cap\mbox{Dom}\left(\left(\overline{\partial}_{q-1}\right)_{G}^{\star}\right)$.
Thus, it is given by an inner product
\begin{equation}\label{7b}
(u,v)_{G^{q}}=Q_{G, q}(N_{q}^{G}u,v)\; .
\end{equation}
$N_{q}^{G}$ is the weighted $\overline{\partial}$-Neumann operator. By definition,
$N_{q}^{G}$ maps $L^{2}_{(0,q)}(\Omega, G^{q})$ continuously into
$\mbox{Dom}(\overline{\partial}_{q})\cap\mbox{Dom}\left(\left(\overline{\partial}_{q-1}\right)_{G}^{\star}\right)$,
a fortiori (by \eqref{seven}) into $L^{2}_{(0,q)}(\Omega, G^{q})$.
It is immediate from \eqref{ww} and \eqref{7b} that $N_{q}^{G}$
inverts $\Box_{q}^{G}$. Denote by $N_{q}$ the $\overline{\partial}$-Neumann operator in the Euclidean metric.
\begin{theorem}\label{compactnessmetric}
Let $\Omega$ be a bounded pseudoconvex domain in $\mathbb{C}^{n}$, $1 \leq q \leq n$. Then $N_{q}^{G}$ is compact if and only if $N_{q}$ is compact.
\end{theorem}
\begin{proof}
Both $N_{q}$ and $N_{q}^{G}$ can be expressed in terms of the
canonical solution operators to $\overline{\partial}$:
\begin{equation}\label{8a}
N_{q} =
\left(\left(\overline{\partial}_{q-1}\right)^{\star}N_{q}\right)^{\star}
\left(\left(\overline{\partial}_{q-1}\right)^{\star}N_{q}\right) +
\left(\left(\overline{\partial}_{q}\right)^{\star}N_{q+1}\right)
\left(\left(\overline{\partial}_{q}\right)^{\star}N_{q+1}\right)^{\star}
\; ,
\end{equation}
and
\begin{equation}\label{8b}
N_{q}^{G}=\left(\left(\overline{\partial}_{q-1}\right)_{G}^{\star}N_{q}^{G}\right)_{G}^{\star}
\left(\left(\overline{\partial}_{q-1}\right)_{G}^{\star}N_{q}^{G}\right)+
\left(\left(\overline{\partial}_{q}\right)_{G}^{\star}N_{q+1}^{G}\right)
\left(\left(\overline{\partial}_{q}\right)_{G}^{\star}N_{q+1}^{G}\right)_{G}^{\star}
\; .
\end{equation}
For $N_{q}$, this is a well known fact, see \cite{MR0461588}, p.55, \cite{Ra84}; for
$N_{q}^{G}$ the proof is the same. Denote by $P_{q}^{G}$ the
orthogonal projection from $L_{0,q}^{2}(\Omega,G^{q})$ onto
ker$(\overline{\partial}_{q})$. Since
$(\overline{\partial}_{q-1})_{G}^{\star}N_{q}^{G}$ annihilates
ker$(\overline{\partial}_{q})^{\perp_{G^{q}}}$, we have
\begin{equation}\label{8c}
\left(\overline{\partial}_{q-1}\right)_{G}^{\star}N_{q}^{G}=
\left(\overline{\partial}_{q-1}\right)_{G}^{\star}N_{q}^{G}P_{q}^{G}
\; .
\end{equation}
Now, if $f$ is a $\overline{\partial}$-closed $(0,q)$-form, then
$\overline{\partial}^{\star}N_{q}f$ and
$(\overline{\partial}_{q-1})_{G_{q}}^{\star}N_{q}^{G}f$ are both
solutions of the equation $\overline{\partial}u=f$; orthogonal in
the respective inner products to ker$(\overline{\partial}_{q-1})$.
Therefore the previous formula implies
\begin{equation}\label{8d}
\left(\overline{\partial}_{q-1}\right)_{G}^{\star}N_{q}^{G}=
(I-P_{q-1}^{G})\overline{\partial}^{\star}N_{q}P_{q}^{G} \; ,
\end{equation}
and (with $q+1$ in place of $q$)
\begin{equation}\label{8e}
\left(\overline{\partial}_{q}\right)_{G}^{\star}N_{q+1}^{G}=
(I-P_{q}^{G})\overline{\partial}^{\star}N_{q+1}P_{q+1}^{G} \; .
\end{equation}
Analogously,
\begin{equation}\label{8f}
\overline{\partial}^{\star}N_{q}=
(I-P_{q-1})\left(\overline{\partial}_{q-1}\right)_{G}^{\star}N_{q}^{G}P_{q}
\; ,
\end{equation}
and
\begin{equation}\label{8g}
\overline{\partial}^{\star}N_{q+1}=
(I-P_{q})\left(\overline{\partial}_{q}\right)_{G}^{\star}N_{q+1}^{G}P_{q+1}
\; .
\end{equation}
By using the fact that $A^{\star}A+BB^{\star}$ is compact if and
only if $A$ and $B$ are compact, the above identities, and the fact
that composition with bounded operators (projections in our case)
preserves compactness, we obtain the theorem.
\end{proof}

For the rest of this section, we assume that $\Omega$ is also ($C^{\infty}$) smooth. A subelliptic estimate of order $\varepsilon > 0$  is said to hold for the $\overline{\partial}$-Neumann problem, if
\begin{equation}
\norm{u}_{G^{q},\varepsilon}^{2}\leq C\left(\norm{\overline{\partial}u}_{G^{q+1}}^{2}+\norm{\overline{\partial}^{\star}u}_{G^{q-1}}^{2}\right)
, \; u \in C_{(0,q)}^{\infty}(\overline{\Omega})\cap\mbox{Dom}\left(\left(\overline{\partial}_{q-1}\right)_{G}^{\star}\right) ,
\end{equation}
where the norm on the left hand side is the $L^{2}$-Sobolev norm of
order $\varepsilon$. We remark that integer Sobolev norms are defined as weighted $L^{2}$-norms of derivatives of forms (acting coefficientwise), and the noninteger ones are then obtained by interpolation.
A subelliptic estimate holds if and only if $N_{q}^{G}$ maps $L_{(0,q)}^{2}(\Omega)$ continuously to $W_{(0,q)}^{2\varepsilon}(\Omega)$.  The proof in the weighted case is the same as in the Euclidean case.
\begin{theorem}\label{subellipticmetric}
Let $\Omega$ be a smooth bounded pseudoconvex domain in $\mathbb{C}^{n}$, $1 \leq q \leq n$, $\varepsilon > 0$. Then $N_{q}^{G}$ is subelliptic of order $2\varepsilon$ if and only if $N_{q}$ is subelliptic of order $2\varepsilon$.
\end{theorem}
\begin{proof}
By the Riesz representation theorem, there is an isomorphism $T_{q}^{G}:L_{(0,q)}^{2}(\Omega)\longrightarrow L_{(0,q)}^{2}(\Omega,G^{q})$ such that
\begin{equation}\label{100}
(u,v) = (T_{q}^{G}u,v)_{G^{q}} \; .
\end{equation}
$T_{q}^{G}$ can also be computed directly from \eqref{eq1}; of importance for us is the fact that the coefficients of $T_{q}^{G}u$ are linear combinations, whose coefficients are functions in $C^{\infty}(\overline{\Omega})$, of the coefficients of $u$. A direct computation shows that $u\in\mbox{Dom}(\overline{\partial}_{q-1}^{\star})$ if and only if $T_{q}^{G}u\in\mbox{Dom}\left(\left(\overline{\partial}_{q-1}\right)_{G_{q}}^{\star}\right)$, and that
\begin{eqnarray}\label{commutativity}
\left(\overline{\partial}_{q-1}\right)_{G}^{\star}T_{q}^{G}u=T_{q-1}^{G}\left(\overline{\partial}_{q-1}\right)^{\star}u \; .
\end{eqnarray}

First assume that there is a subelliptic estimate of order $\varepsilon>0$ in the weighted norm.
Let $u\in\mbox{Dom}(\overline{\partial}_{q})\cap\mbox{Dom}(\overline{\partial}_{q}^{\star})$.
Then $T_{q}^{G}u\in \mbox{Dom}\left(\left(\overline{\partial}_{q-1}\right)_{G_{q}}^{\star}\right)$, and we have
\begin{multline}\label{proofss}
\norm{u}_{\varepsilon}^{2} \; \lesssim \; \norm{T_{q}^{G}u}_{\varepsilon,G^{q}}^{2} + \|u\|^{2}  \\
\lesssim \; \norm{\overline{\partial}_{q}T_{q}^{G}u}_{G^{q+1}}^{2}+
\norm{\left(\overline{\partial}_{q-1}\right)_{G}^{\star}T_{q}^{G}u}_{G^{q-1}}^{2} + \|u\|^{2}\;\;\;\;\;\;\;\;\;\;\;\;\;\;\;\;\; \\
\;\;\; \lesssim \sideset{}{'}\sum_{|K|=q}\sum_{j=1}^{n}\norm{\frac{\partial u_{K}}{\partial \overline{z}_{j}}}^{2}+\norm{u}^{2}+\norm{T_{q-1}^{G}\overline{\partial}_{q-1}^{\star}u}^{2}\;\;\;\;\;\;\;\;\;\;\;\; \;\;\;\;\\
\lesssim \sideset{}{'}\sum_{|K|=q}\sum_{j=1}^{n}\norm{\frac{\partial
u_{K}}{\partial
\overline{z}_{j}}}^{2}+\norm{u}^{2}+\norm{\overline{\partial}_{q-1}^{\star}u}^{2}
 \lesssim \norm{\overline{\partial}_{q}u}^{2}+\norm{\overline{\partial}_{q-1}^{\star}u}^{2}.
\end{multline}
The third inequality results from the form of $T_{q}^{G}u$ pointed
out above, and the last inequality is from the Kohn-Morrey formula
(see e.g. \cite{MR1800297}). \eqref{proofss} is the desired
subelliptic estimate for the unweighted metric.

The above argument relies (among other things) on the Kohn-Morrey inequality. Instead of attempting to derive a (complicated) version for the case of general metrics $G^{q}$, $1 \leq q \leq n$, we give a different argument for the proof of the reverse direction (this argument could also be used for the first direction). When $q=1$, this argument will involve the $\overline{\partial}$-Neumann operators $N_{0}$ and $N_{0}^{G}$; a reference is \cite{MR1800297} Theorem $4.4.3$.

Let $u\in \mbox{Dom}(\overline{\partial}_{q})\cap
\mbox{Dom}\left(\left(\overline{\partial}_{q-1}\right)_{G_{1}}^{\star}\right)\subset L_{(0,q)}^{2}(\Omega,G^{q})$. Then
\begin{eqnarray}\label{q1}
u=\left(\overline{\partial}_{q}\right)_{G}^{\star}N_{q+1}^{G}(\overline{\partial}_{q}u)+
\overline{\partial}_{q-1}N_{q-1}^{G}\left(\left(\overline{\partial}_{q-1}\right)_{G}^{\star}u\right).
\end{eqnarray}
Because $(I-P_{1}^{G})$ projects onto the range of $\left(\overline{\partial}_{q}\right)_{G_{q+1}}^{\star}$, we have
\begin{eqnarray}\label{q2}
\left(\overline{\partial}_{q}\right)_{G_{q+1}}^{\star}N_{q+1}^{G}(\overline{\partial}u)=\left(I-P_{q}^{G}\right)\overline{\partial}_{q}^{\star}N_{q+1}(\overline{\partial}_{q}u).
\end{eqnarray}
For the second term on the right of (\ref{q1}), we use that $\overline{\partial}_{q-1}N_{q-1}^{G}=N_{q}^{G}\overline{\partial}_{q-1}$ and then \eqref{q2} with $q-1$ in place of $q$ to obtain
\begin{multline}\label{q3}
\overline{\partial}_{q-1}N_{q-1}^{G}=N_{q}^{G}\overline{\partial}_{q-1}=
\left(\left(\overline{\partial}_{q-1}\right)_{G}^{\star}N_{q}^{G}\right)_{G}^{\star} \\ = \left(\left(I-P_{q-1}^{G}\right)\overline{\partial}_{q-1}^{\star}N_{q}P_{q}^{G}\right)_{G}^{\star}
 = P_{q}^{G}\left(\overline{\partial}_{q-1}^{\star}N_{q}\right)_{G}^{\star}\left(I-P_{q-1}^{G}\right).
\end{multline}
Thus, by using (\ref{q2}) and (\ref{q3}), we write (\ref{q1}) as
\begin{eqnarray}\label{q4}
u= \left(I-P_{q}^{G}\right)\overline{\partial}^{\star}N_{q+1}(\overline{\partial}_{q}u)+P_{q}^{G}\left(\overline{\partial}_{q-1}^{\star}N_{q}\right)_{G}^{\star}\left(\left(\overline{\partial}_{q-1}\right)_{G_{q}}^{\star}u\right)\; .
\end{eqnarray}
We have used that $(I-P_{q-1}^{G})\left(\overline{\partial}_{q-1}\right)_{G}^{\star}u=
\left(\overline{\partial}_{q-1}\right)_{G}^{\star}u$.

A subelliptic estimate associated with the metric $G$ will follow if we can show the following two mapping properties (as continuous maps):
\begin{itemize}
\item[(i)] $P_{q}^{G}:W_{(0,q)}^{\varepsilon}(\Omega,G^{q})\longrightarrow W_{(0,q)}^{\varepsilon}(\Omega,G^{q})$  ,

\item[(ii)] $\left(\overline{\partial}_{q-1}^{\star}N_{q}\right)_{G}^{\star}:
L_{(0,q)}^{2}(\Omega,G^{q})\longrightarrow W_{(0,q-1)}^{\varepsilon}(\Omega,G^{q})$  .
\end{itemize}

For (i), note that $N_{q}$ is compact (since it is subelliptic); hence so is $N_{q}^{G}$, by Theorem \ref{compactnessmetric}. This implies that $P_{q}^{G}= \overline{\partial}\overline{\partial}^{*}_{G}N^{G}_{q}$ preserves the Sobolev spaces: the proof is analogous to the unweighted case (\cite{MR1800297}, Theorem 6.2.2). Note that the form $Q_{G,q}$ is also covered by the results in \cite{MR0181815}.

As for (ii), note that $W_{(0,q)}^{\varepsilon}(\Omega,G)=W_{(0,q)}^{\varepsilon}(\Omega)$ and $\left(W_{(0,q)}^{\varepsilon}(\Omega)\right)^{\star}=W_{(0,q)}^{-\varepsilon}(\Omega)$, since $0\leq\varepsilon\leq 1/2$; the spaces $W^{\varepsilon}(\Omega)$ and $W_{0}^{\varepsilon}(\Omega)$ coincide for $0\leq\varepsilon\leq 1/2$ (with equivalent norms), see Theorem $11.1$ in \cite{MR0350177}. This duality similarly holds for the weighted spaces, with the weighted $L^{2}$ pairing. Thus, the statement in (ii) is equivalent to having an extension of $\overline{\partial}^{*}N_{q}$ as a continuous map
\begin{eqnarray}\label{q4}
\overline{\partial}^{\star}N_{q}:W_{(0,q)}^{-\varepsilon}(\Omega)\longrightarrow L_{(0,q-1)}^{2}(\Omega) \;.
\end{eqnarray}
This is a well know consequence of subellipticity of $N_{q}$, as follows. $L_{(0,q)}^{2}(\Omega)$ is dense in $W_{(0,q)}^{-\varepsilon}(\Omega)$, so to prove (\ref{q4}) we let $u \in L^{2}_{(0,q)}(\Omega)$ and estimate
\begin{multline}\label{q5}
\norm{\overline{\partial}^{\star}N_{q}u}^{2}+\norm{\overline{\partial}N_{q}u}^{2} = \left(\overline{\partial}\overline{\partial}^{\star}N_{q}u,N_{q}u\right)+
\left(\overline{\partial}^{*}\overline{\partial}N_{q}u, N_{q}u\right)\\
= \left(u,N_{q}u\right)
\lesssim \norm{u}_{-\varepsilon}\norm{N_{q}u}_{\varepsilon}
\leq (\mbox{s.c.})\norm{N_{q}u}_{\varepsilon}^{2}+(\mbox{l.c.})\norm{u}_{-\varepsilon}^{2}\\
\lesssim (\mbox{s.c.})\left(\norm{\overline{\partial}^{\star}N_{q}u}^{2}+\norm{\overline{\partial}N_{q}u}^{2}\right)+(\mbox{l.c.})\norm{u}_{-\varepsilon}^{2} \;.
\end{multline}
The last inequality comes from the subelliptic estimate associated with the Euclidean metric. The first term on the right in the last inequality can be absorbed into the left side of \eqref{q5} to obtain
\begin{eqnarray}\label{q5a}
\norm{\overline{\partial}^{\star}N_{q}u}^{2}+\norm{\overline{\partial}N_{q}u}^{2}\lesssim\norm{u}_{-\varepsilon}^{2}.
\end{eqnarray}
This was for $u\in L_{(0,q)}^{2}(\Omega)$. By density both $\overline{\partial}^{\star}N_{q}$ and $\overline{\partial}N_{q}$ extend to continuous operators from $W_{(0,q)}^{-\varepsilon}(\Omega)$ to $L^{2}_{(0,q-1)}(\Omega)$ and $L^{2}_{(0,q+1)}(\Omega)$, respectively. In particular, (\ref{q4}) holds. This completes the proof of Theorem \ref{subellipticmetric}.
\end{proof}

\section{Obstructions to Compactness}
Let $\Omega $ be a bounded pseudoconvex domain in $\mathbf{C}^{n}$. Recall that $N_{q}$ is compact if and only if there is a so called compactness estimate: for every  $\varepsilon>0$ there is a constant
 $C_{\varepsilon}>0$ such that the estimate
\begin{eqnarray}\label{eq0}
\norm{u}^{2}\leq \varepsilon\left(\norm{\overline{\partial}u}^{2}+\norm{\overline{\partial}^{\star}u}^{2}\right)+C_{\varepsilon}\norm{u}_{-1}^{2}
\end{eqnarray}
is valid for all $u\in \mbox{Dom}(\overline{\partial})\cap\mbox{Dom}(\overline{\partial}^{\star})\subset L_{(0,q)}^{2}(\Omega)$ (\cite{MR1912737}, Lemma 1.1, \cite{Vienna}, Proposition 3.2).

A function $f\in
C(\overline{\Omega})$ is called a compactness multiplier on $\Omega$
if for every $\varepsilon>0$ there is a constant
$C_{\varepsilon,f}>0$ such that the estimate
\begin{eqnarray}\label{eq:no1a}
\norm{fu}^{2} \leq \varepsilon ( \norm{\overline{\partial}u}^{2}+\norm{\overline{\partial}^{\star}u}^{2})+ C_{\varepsilon,f}\norm{u}_{-1}^{2}
\end{eqnarray}
is valid for all $u\in \mbox{Dom}(\overline{\partial})\cap
\mbox{Dom}(\overline{\partial}^{\star})\subset
L_{(0,q)}^{2}(\Omega)$. Note that $f \in C(\overline{\Omega})$ is a compactness multiplier if and only if the multiplication operator $M_{f}: u \rightarrow fu$ from $\text{Dom}(\overline{\partial}) \cap \text{Dom}(\overline{\partial}^{*})$, equipped with the graph norm, to $L^{2}_{(0,q)}(\Omega)$ is compact. Namely, in terms of the graph norm $\|u\|_{graph}$, estimate \eqref{eq:no1a} says that $\|M_{f}u\|^{2} \leq \varepsilon \|u\|_{graph}^{2} + C_{\varepsilon, f}\|u\|_{-1}^{2}$. Because $L^{2}(\Omega)$ embeds compactly into $W^{-1}(\Omega)$, having this inequality for all $\varepsilon > 0$ characterizes compactness of the operator $M_{f}: \text{Dom}(\overline{\partial}) \cap \text{Dom}(\overline{\partial}^{*}) \rightarrow L^{2}_{(0,q)}(\Omega)$, see for example \cite{D'Angelo}, Proposition V.2.3, \cite{MR1934357}, Lemma 2.1, \cite{Vienna}, Lemma 3.3.

The basic properties of compactness multipliers are rather more
elementary than the corresponding facts for subelliptic multipliers (\cite{D'Angelo93}, \cite{MR1748604}).
Let $J^{q}$ be the set of the compactness multipliers
defined as above, associated with $(0,q)$ forms, $1\leq q\leq n$. Denote by $A_{q}$ the common zero set of the elements of $J^{q}$. $A_{q}$ is compact, and by interior elliptic regularity of the complex $\overline{\partial}\oplus\overline{\partial}^{*}$, $A_{q} \subseteq b\Omega$. More precisely, any $\varphi \in C^{\infty}_{0}(\Omega)$ is a compactness multiplier: if $u \in \text{Dom}(\overline{\partial}) \cap \text{Dom}(\overline{\partial}^{*})$, then $\varphi u$ has components in $W^{1}_{0}(\Omega)$. The latter space embeds compactly into $L^{2}(\Omega)$, so that $M_{\varphi}$ is indeed compact.
\begin{proposition}\label{lem2}
Let $\Omega$ be a bounded pseudoconvex domain in $\C^{n}$. The set of compactness multipliers $J^{q}$
is a closed ideal in $C(\overline{\Omega})$, and so equals $\{f \in C(\overline{\Omega}) \;|\; f \equiv 0 \;\text{on}\;A_{q}\}$.
\end{proposition}
\begin{proof}
It is easy to see that $hg$ is a compactness multiplier whenever $g$
is; $\norm{(hg)u}^{2}\leq \left (\sup_{z \in \overline{\Omega}} |h(z)| \right )\norm{gu}^{2}$. Thus,
$J^{q}$ is closed under multiplication by elements of
$C(\overline{\Omega})$. The sum of two compactness multipliers is a
compactness multiplier: $\norm{(g+f)u}^{2}\leq
2(\norm{gu}^{2}+\norm{fu}^{2})$. So $J^{q}$ is an ideal
of $C(\overline{\Omega})$.

To see that $J^{q}$ is closed under the \textit{sup}-norm, observe that the operator norm of $M_{f}$
(as an operator from
$\text{Dom}(\overline{\partial}) \cap \text{Dom}(\overline{\partial}^{*}) \rightarrow L^{2}_{(0,q)}(\Omega)$) is
dominated by $\sup_{z \in \overline{\Omega}}|f(z)|$. Indeed, we have
\begin{equation}\label{eq1b}
\|M_{f}u\|^{2} \leq \left (\sup_{z \in \overline{\Omega}}|f(z)| \right )^{2}\|u\|^{2} \leq \frac{D^{2}e}{q}\left (\sup_{z \in \overline{\Omega}}|f(z)| \right )^{2}\left ( \|\overline{\partial}u\|^{2} + \|\overline{\partial}^{*}u\|^{2} \right ) \; ,
\end{equation}
where $D$ is the diameter of $\Omega$. The second inequality is the fundamental $L^{2}$ estimate for the $\overline{\partial}$-complex dating back to H\"{o}rmander (\cite{MR0179443}, \cite{MR1800297}, \cite{Vienna}). Therefore, if $f \in C(\overline{\Omega})$ is a uniform limit of a sequence of compactness multipliers $\{f_{n}\}_{n=1}^{\infty}$, then the corresponding compact multiplication operators $M_{f_{n}}$ converge in operator norm to $M_{f}$. Consequently, $M_{f}$ is compact as well, and $f$ is a compactness multiplier.

Finally, any closed ideal in $C(\overline{\Omega})$ is the full ideal generated by the zero set. For this elementary fact, see for example \cite{MR1216137}, Theorem 2.1. In our situation, this fact is also easily established directly.
\end{proof}

\emph{Remark 2}: Any function in $C(\overline{\Omega})$ that vanishes on the boundary is thus a compactness multiplier.

\emph{Remark 3}: When $\Omega$ is a smooth domain, the set $A_{q}$ is a subset of the set of boundary points of infinite type. This is immediate because a subelliptic pseudolocal estimate holds near a point of finite type (\cite{Catlin87}).

\emph{Remark 4}: If the set $A_{q}$ is not empty, it cannot be `too small'. In particular,
it cannot satisfy property $(P_{q})$ (see \cite{MR740870} for $(P_{1})$, \cite{MR1912737}
and \cite{Vienna} for $(P_{q})$, $q \geq 1$). The proof of this fact is analogous to the
proof that compactness is a local property, see Lemma 1.2 in \cite{MR1912737}, Proposition 3.4
in \cite{Vienna}; essentially the same argument also occurs in the first part of the proof of
Theorem \ref{theorem_compactness} below. One shows indirectly that $A_{q}$ satisfying $(P_{q})$
implies a compactness estimate by writing a form $u$ as $u_{1} + u_{2}$, with $u_{1}$ supported
near $A_{q}$, and $u_{2}$ supported away from $A_{q}$. Then $u_{1}$ is estimated by using the estimate
$\sum^{\prime}_{K}\sum_{j,k}\int_{\Omega}(\partial^{2}\lambda/\partial z_{j}\partial\overline{z_{k}})u_{j,K}
\overline{u_{k,K}} \leq e(\|\overline{\partial}u\|^{2} + \|\overline{\partial}^{*}u\|^{2})$
(\cite{MR1748601}, p.83, \cite{Vienna}, Corollary 1.12) in the usual way. $u_{2}$ is estimated
via $u_{2} = \varphi u_{2}$, where $\varphi$ is supported away from $A_{q}$ and so is a compactness
multiplier. Details of this argument are in \cite{celik} and in the first part of the proof of Theorem \ref{theorem_compactness} below. In particular, $A_{1}$ cannot have two dimensional
Hausdorff measure zero, as such sets satisfy $(P_{1})$ (\cite{boas88, MR894582}), nor can it be contained
in a subvariety of the boundary of holomorphic dimension zero (\cite{MR894582}, Proposition 12).

\bigskip

We do not attempt to develop a serious theory of compactness multipliers here; in particular, we ignore questions relating to the algorithmic point of view in \cite{MR512213, MR1748604, nicoara}. Instead, we determine the sets $A_{q}$ for two classes of domains.

Denote by $\left\{f_{\alpha}(\mathbb{D}^{q})\right\}_{\alpha\in\Lambda}$ the family of $q$-dimensional
analytic polydiscs in the boundary of $\Omega$. That is, $f_{\alpha}$ is holomorphic on the $q$-dimensional
unit polydisc $\mathbb{D}^{q}$ and continuous on its closure, and it maps into $b\Omega$. It was shown by Fu and Straube (\cite{MR1659575})
that on a convex domain, the $\overline{\partial}$-Neumann operator is compact if and only if the boundary contains no $q$-dimensional analytic varieties. This motivates the following theorem.

\begin{theorem}\label{theorem_compactness}
Let $\Omega$ be a bounded convex domain in $\mathbb{C}^{n}$. Then
\begin{equation}\label{eq2a}
A_{q}= \overline{\bigcup_{\alpha\in\Lambda} f_{\alpha}(\mathbb{D}^{q})} \;.
\end{equation}
\end{theorem}

\begin{proof}
We first show that if $P \in b\Omega$ is not in $A_{q}$, then it is not in the right hand side of \eqref{eq2a}.
Choose $r>0$ small enough so that $\Omega_{1} := \Omega \cap B(P,r)$ is a convex domain whose closure does not
intersect $A_{q}$. It suffices to establish a compactness estimate on $\Omega_{1}$: the result of Fu and Straube
mentioned above then implies that the boundary of $\Omega_{1}$ contains no $q$-dimensional analytic variety,
whence $P \notin \overline{\bigcup_{\alpha} f_{\alpha}(\mathbb{D}^{q})}$.

Let $M > 0$. Choose $\varphi_{M} \in
C^{\infty}_{0}(\mathbb{C}^{n})$, $0 \leq \varphi_{M} \leq 1$, and
supported on the set where $-1/M < |z-P|^{2} - r^{2} < 1/M$. Now let
$u \in \text{Dom}(\overline{\partial}) \cap
\text{Dom}(\overline{\partial}^{*})$ on $\Omega_{1}$. Denote by
$\lambda_{M}(z)$ a smooth function that on the support of
$\varphi_{M}$ agrees with $M(|z-P|^{2} - r^{2})$ and otherwise is
between $-1$ and $1$. Note that on the support of $\varphi_{M}$, the
complex Hessian of $\lambda_{M}$ is at least $M$. To estimate the
norm of $\varphi_{M}u$, we use inequality (2-10) from
\cite{MR1748601} which says that
\begin{equation}\label{eq2b}
\sum^{\prime}_{K}\sum_{j,k}\int_{\Omega_{1}}e^{\lambda_{M}}\frac{\partial^{2}\lambda_{M}}{\partial
z_{j}\partial\overline{z_{k}}}(\varphi_{M}u)_{jK}\overline{(\varphi_{M}u)_{kK}}\;
\leq \;\|\overline{\partial}(\varphi_{M}u)\|_{\Omega_{1}}^{2} +
\|\overline{\partial}^{*}(\varphi_{M}u)\|_{\Omega_{1}}^{2} \; .
\end{equation}
A comment is in order. (2-10) in \cite{MR1748601} is stated for sufficiently smooth domains. We make no smoothness
assumptions on $\Omega$ other than what is dictated by convexity ($\Omega$ is Lipschitz). In addition,
$\Omega_{1}$ has a nonsmooth part in the boundary coming from the intersection of $\Omega$ with a small ball.
However, the exhaustion procedure developed in \cite{Pluri} that uses the $\overline{\partial}$-Neumann
operators on a sequence of subdomains allows to forgo any boundary regularity assumptions: (2-10) in
\cite{MR1748601} holds on any bounded pseudoconvex domain. This is part (ii) of Corollary 1.12 in \cite{Vienna}.

For $M$ big enough, we can choose $\chi_{M} \in
C^{\infty}_{0}(B(P,r))$, identically equal to one on a
neighborhood of the part of the support of $(1-\varphi_{M})$ that
lies in $\overline{\Omega_{1}}$. Note that $\chi_{M}$ (continued by zero
outside $B(P,r)$) is a compactness multiplier on $\Omega$ (since
it vanishes on $A_{q}$). Also, the left hand side of \eqref{eq2b}
dominates $qM/e$ times $\|\varphi_{M}u\|_{\Omega_{1}}^{2}$; the
factor $q$ occurs because each term $|u_{J}|^{2}$ arises precisely
$q$ times as a term $|u_{j,K}|^{2}$. Therefore, we have for any
$\varepsilon^{\prime}$
\begin{multline}\label{eq2c}
\|u\|_{\Omega_{1}}^{2} \lesssim
\|\varphi_{M}u\|_{\Omega_{1}}^{2} +  \|\chi_{M}(1-\varphi_{M})u\|_{\Omega}^{2} \\
\lesssim \left (\frac{1}{M} + \varepsilon^{\prime} \right ) \left
(\|\overline{\partial}u\|_{\Omega_{1}}^{2} +
\|\overline{\partial}^{*}u\|_{\Omega_{1}}^{2} +
\|D\varphi_{M}u\|_{\Omega}^{2} \right ) + C_{\varepsilon^{\prime},
M}\|(1-\varphi_{M})u\|_{-1,\Omega}^{2} \; ,
\end{multline}
where $D\varphi_{M}$ denotes a derivative of $\varphi_{M}$. We have
used that $(1-\varphi_{M})u \in \text{Dom}(\overline{\partial}) \cap
\text{Dom}(\overline{\partial}^{*})$ on $\Omega$. To estimate
$\|D\varphi_{M}u\|_{\Omega}^{2}$, we use again that $\chi_{M}$ is a
compactness multiplier on $\Omega$:
\begin{multline}\label{eq2cc}
\|D\varphi_{M}u\|_{\Omega}^{2} =
\|\chi_{M}D\varphi_{M}u\|_{\Omega}^{2} \\
\leq \varepsilon^{\prime}C_{M} \left
(\|\overline{\partial}u\|_{\Omega_{1}}^{2} +
\|\overline{\partial}^{*}u\|_{\Omega_{1}}^{2} +
\|u\|_{\Omega_{1}}^{2} \right ) +
C_{\varepsilon^{\prime}}\|D\varphi_{M}u\|_{-1,\Omega}^{2} \\
\lesssim \varepsilon^{\prime}C_{M} \left
(\|\overline{\partial}u\|_{\Omega_{1}}^{2} +
\|\overline{\partial}^{*}u\|_{\Omega_{1}}^{2} \right ) +
C_{\varepsilon^{\prime}}\|D\varphi_{M}u\|_{-1,\Omega}^{2} \; .
\end{multline}
Because $(1-\varphi_{M})$ and $D\varphi_{M}$ are compactly supported
in $B(P,r) \cap \overline{\Omega}$, the $(-1)$-norms on $\Omega$ on
the right hand side of \eqref{eq2c} and \eqref{eq2cc} are dominated
by the corresponding $(-1)$-norms on $\Omega_{1}$ and hence by
$\|u\|_{-1, \Omega_{1}}$. Therefore, the desired compactness
estimate on $\Omega_{1}$ results from \eqref{eq2c} and \eqref{eq2cc}
upon taking $M$ big enough and then $\varepsilon^{\prime}$ small
enough.

For the other direction, assume that $P \notin \overline{\bigcup_{\alpha} f_{\alpha}(\mathbb{D}^{q})}$. Choose $r>0$ small enough so that the closure of $\Omega_{1} := B(P,r) \cap \Omega$ is disjoint from $\overline{\bigcup_{\alpha} f_{\alpha}(\mathbb{D}^{q})}$. $\Omega_{1}$ is a convex domain without $q$-dimensional varieties in the boundary. Again by the Fu-Straube result mentioned above, the $\overline{\partial}$-Neumann operator on $(0,q)$-forms on $\Omega_{1}$ is compact. Therefore, for a smooth function $\varphi$ supported near $P$, we have for any $\varepsilon > 0$
\begin{multline}\label{eq2d}
\|\varphi u\|_{\Omega}^{2} = \|\varphi u\|_{\Omega_{1}}^{2} \\
\leq \varepsilon \left (\|\overline{\partial}(\varphi u)\|_{\Omega_{1}}^{2} + \|\overline{\partial}^{*}(\varphi u)\|_{\Omega_{1}}^{2} \right ) + C_{\varepsilon}\|\varphi u\|_{-1, \Omega_{1}}^{2} \\
\lesssim \varepsilon C_{\varphi}\left (\|\overline{\partial}
u\|_{\Omega}^{2} + \|\overline{\partial}^{*} u\|_{\Omega}^{2} +
\|u\|_{\Omega}^{2} \right ) +
C_{\varepsilon, \varphi}\|u\|_{-1,\Omega_{1}}^{2} \\
\lesssim \varepsilon C_{\varphi} \left (\|\overline{\partial}
u\|_{\Omega}^{2} + \|\overline{\partial}^{*} u\|_{\Omega}^{2} \right
) + C_{\varepsilon, \varphi}\|u\|_{-1,\Omega}^{2} \; .
\end{multline}
In the last inequality, we have used the easily verified inequality
$\|u\|_{-1,\Omega_{1}} \leq \|u\|_{-1,\Omega}$. \eqref{eq2d} shows
that any such $\varphi$ is a compactness multiplier on $\Omega$.
Choosing a $\varphi$ that does not vanish at $P$ shows that $P
\notin A_{q}$. This completes the proof of Theorem
\ref{theorem_compactness}.
\end{proof}

We now turn to complete Hartogs domains in $\mathbb{C}^{2}$. A complete
Hartogs domain $\Omega$ in $\C^{2}$ is defined by $\abs{w}<
e^{-\phi(z)}$ for $z\in U$, where $U$ is a domain in $\mathbb{C}$
and $\phi(z)$ is an upper semi-continuous function on $U$. $\Omega$ is
pseudoconvex if and only if $\phi(z)$ is subharmonic. If $\phi$ is
at least $C^{2}$, then a computation shows that the weakly pseudoconvex boundary points
$(z,w)$ with $w\neq 0$ are those where $\Delta\phi(z)=0$ (see also \cite{MR847923}, p. 100).

On a smooth bounded pseudoconvex (not necessarily complete) Hartogs
domain compactness of the $\overline{\partial}$-Neumann operator is
equivalent to Catlin's property ($P$) (\cite{MR2166176}).
Additionally, if the domain is also complete and the boundary points
with $w=0$ are strictly pseudoconvex, then both of the above
conditions (compactness and property ($P$)) are equivalent to the
following: the projection of the weakly pseudoconvex boundary points
into the $z$-plane has empty fine interior (\cite{MR894582}, p.310,
\cite{Vienna}, Lemma 3.19). Recall that the fine topology is the
smallest topology that makes all subharmonic functions continuous;
see, e.g. \cite{MR0261018, MR1801253} for properties of this
topology. It is strictly larger than the Euclidean topology, and
there exist compact sets with empty Euclidean interior, but nonempty
fine interior, see \cite{MR1801253} example $7.9.3$. Of importance
here will be the following fact, which explains why the fine
topology plays a role in our context: a compact subset of the plane
satisfies property $(P)$ if and only if it has empty fine interior
(\cite{MR894582}, Proposition 1.11, \cite{Vienna}, Proposition
3.17). We will use the notation $Int_{f}(K)$ for the fine interior
points of the set $K$, and $\overline{B}^{E}$ for the
Euclidean closure of a set $B$.

Denote by $\pi$ the projection $\pi: \mathbb{C}^{2} \rightarrow
\mathbb{C}$. $(z,w) \rightarrow z$. Let $K$ be the projection of the
weakly pseudoconvex boundary points. When the boundary points
with $w=0$ are strictly pseudoconvex, $K =
\{z \in U\;|\;\Delta\phi(z)=0\}$, and $K$ is a compact subset of the
base domain $U$. On such a domain we have the following
characterization of $A:=A_{1}$.

\begin{theorem}\label{theorem_compactnessA}
Let $\Omega$ be a smooth bounded complete pseudoconvex Hartogs
domain in $\C^{2}$. Assume that the boundary points of the form
$(z,0)$ are strictly pseudoconvex. Then
$$A=\overline{\pi^{-1}\left(Int_{f}(K)\right)}^{E}.$$
\end{theorem}

\begin{proof}
The structure of the proof is the same as in the proof of Theorem
\ref{theorem_compactness}, but the details change. We first show
that if $P\in b\Omega \setminus
\overline{\pi^{-1}\left(Int_{f}(K) \right)}^{E}$, then $P \in
b\Omega \setminus A$. Choose $r>0$ small enough so that
$\pi(B(P,r))$ is disjoint from $\overline{Int_{f}(K)}^{E}$ and set
$\Omega_{1}:=B(P,r) \cap \Omega$. We will show that $b\Omega_{1}$
satisfies property $(P_{1})$, so that the $\overline{\partial}$-Neumann 
operator on $\Omega_{1}$ is compact. The rest of the argument then follows
that in the second part of the proof of Theorem
\ref{theorem_compactness}.

We write $b\Omega_{1}$ as a countable union of compact sets, all of
which satisfy property $(P_{1})$; then so does $b\Omega_{1}$
(\cite{MR894582}, Proposition 1.9, \cite{Vienna}, Corollary 3.14). The first set
is $bB(P,r) \cap \overline{\Omega}$. The second set consists of the
set $W$ of weakly pseudoconvex boundary points of $\Omega$ that are
contained in $\overline{B(P,r)}$. Note that $\pi(W)$ has empty fine
interior, hence satisfies property $(P_{1})$ in the plane.
Therefore, $W$ satisfies property $(P_{1})$, by \cite{MR894582},
Proposition 1.10 (the image $\pi(W)$ does, and the fiber over each point is
a circle and so also satisfies $(P_{1})$). Finally, write the set of
strictly pseudoconvex boundary points of $\Omega$ as a countable
union of compact sets $K_{n}$; they then satisfy property $(P_{1})$.
Thus so do the compacts $K_{n} \cap \overline{B(P,r)}$. The union of
these sets together with $W$ and the first set above equals
$b\Omega_{1}$, and we are done.

Now let $P \in b\Omega \setminus A$. We will construct a Hartogs domain $\Omega_{1}$ that shares a (rotationally invariant) piece of boundary with $\Omega$ that contains $P$ and does not intersect $A$. In addition, $\Omega_{1}$ will be strictly pseudoconvex off of that shared piece. Observe that $A$ is invariant under rotations in $w$: pullbacks commute with $\overline{\partial}$ and, because the rotations induce isometries on $L^{2}_{(0,1)}(\Omega)$, also with $\overline{\partial}^{*}$ (this observation was exploited in \cite{BCS}). So $f$ is a compactness multiplier if and only if $f_{\theta}(z,w):=f(z,e^{i\theta}w)$ is. Therefore, $P \in b\Omega \setminus A$ implies $P_{1}:=\pi(P) \notin \pi(A)$. Choose radii $0<r_{1}<r_{2}<r_{3}$ such that $\overline{D(P_{1},r_{3})} \cap \pi(A)=\emptyset$. $\Omega_{1}$ is going to be over the base $D(P_{1},r_{3})$. Then choose $\varphi(z)\in C_{0}^{\infty}(D(P_{1},r_{3}))$ such that $\varphi(z)\equiv 1$ on $\overline{D(P_{1},r_{2})}$. Now set $\psi(z):=\phi(z)\varphi(z)+h(z)$ on $D(P_{1},r_{3})$, where $h(z)$ is a smooth radially symmetric subharmonic function on $D(P_{1},r_{3})$, $h(z)=0$ on $\overline{D(P_{1},r_{1})}$, and equal to $-\frac{1}{2}\log(r_{3}^{2}-|z-P_{1}|^{2})$ when $|z-P_{1}|$ is close to $r_{3}$. Such a function can be chosen to be increasing and concave up, and to have its second (radial) derivative as big as we wish on a given compact subset of $D(P_{1},r_{3})\setminus \overline{D(P_{1},r_{1})}$, in particular on $\left\{\Delta(\phi\varphi) \leq 0\right\} \cap supp (\varphi) \cap \{|z| \geq r_{2}\}$. That means that he Laplacian of $h$ on this set can be made as big as we wish. Making this Laplacian big enough ensures that $\psi(z)$ is subharmonic, and
$$\Omega_{1}:=\lbrace (z,w)\in\C^{2}\ |\ z\in D(P_{1},r_{3}), \ \abs{w}<e^{-\psi(z)}\rbrace$$
is a smooth bounded pseudoconvex complete Hartogs domain which near the boundary of its base looks like a ball. This construction comes from \cite{MR1912737}, proof of Theorem 4.2; see also \cite{Vienna}, proof of Theorem 3.25.

We claim that the portion of the boundary of $\Omega_{1}$ that sits over $\{r_{1} \leq |z-P_{1}| \leq r_{3}\}$ satisfies property $(P_{1})$. Parts over compact subsets of $\{r_{1} < |z-P_{1}| < r_{3}\}$ are compact subsets of the strictly pseudoconvex part of the boundary and so satisfy property $(P_{1})$. The part corresponding to $\pi^{-1}(\{|z|=r_{1}\})$ satisfies $(P_{1})$ for the same reason that the set $W$ in the first part of the proof did (i.e. by \cite{MR894582}, Proposition 1.10). The circle $\{|z|=r_{3}\}$ also satisfies $(P_{1})$. Thus the portion of the boundary we are interested in is the countable union of compact sets that satisfy property $(P_{1})$, and the claim is established (again as in the first part of the proof).

Using the claim from the previous paragraph together with the fact that for each boundary point that is common to $\Omega_{1}$ and $\Omega$ (these are the boundary points of $\Omega_{1}$ over the set $\{|z| \leq r_{1}\}$) there exists a compactness multiplier on $\Omega$ that does not vanish at the point, one can follow the first part of the proof of Theorem \ref{theorem_compactness} to show that the $\overline{\partial}$-Neumann problem on $\Omega_{1}$ satisfies a compactness estimate. By the result of Christ and Fu \cite{MR2166176}, the boundary of $\Omega_{1}$ satisfies property $(P_{1})$, and by our discussion above, the projection of its weakly pseudoconvex boundary points onto the $z$-plane has empty fine interior. Therefore (because $\Omega$ and $\Omega_{1}$ share an open piece of boundary near $P$), $P_{1} \notin \overline{int_{f}(K)}^{E}$, and consequently $P \in b\Omega \setminus \overline{\pi^{-1}\left(Int_{f}(K)\right)}^{E}$. This completes the proof of Theorem \ref{theorem_compactnessA}.
\end{proof}

\end{document}